\documentclass[11pt]{article}

\usepackage[T1]{fontenc}
\usepackage{times}
\usepackage{rawfonts}
\usepackage{latexsym}
\usepackage{amsmath}
\usepackage{amsthm}
\usepackage{amsfonts}
\usepackage{amssymb}
\usepackage{xspace}
\usepackage[pdftex]{hyperref}
\usepackage{color}
\newtheorem{theorem}{Theorem}

\newtheorem{proposition}{Proposition}
\theoremstyle{remark}
\newtheorem{rem}{Remark}
\newtheorem{example}{Example}

\theoremstyle{definition}
\newtheorem{definition}[theorem]{Definition}

\title{\bf On ill-posedness concepts, stable solvability and saturation}
\newcommand{\xad}{x_\alpha^\delta}

\newcommand{\xdag}{x^\dagger}
\newcommand{\yd}{y^\delta}

\author{{\sc Bernd Hofmann} \footnote{Faculty of Mathematics, Chemnitz University of Technology, 09107 Chemnitz, Germany.}
$\quad$ and $\quad$ {\sc Robert Plato}\footnote{Department of Mathematics, University of Siegen,
Walter-Flex-Str.~3, 57068 Siegen, Germany.}}

\newcommand{\refeq}[1]{(\ref{eq:#1})}

\newcommand{\re}{ {\rm \bf Re \, }}
\newcommand{\skp}[2]{\langle #1,  #2 \rangle}

\newcommand{\mypsi}{\varphi}

\newcommand{\uk}[1][k]{u_{#1}}
\newcommand{\wk}[1][k]{w^{(#1)}}
\newcommand{\wnk}[1][k]{w^{(#1)}_n}
\newcommand{\vk}[1][k]{v^{(#1)}}
\newcommand{\vnk}[1][k]{v^{(#1)}_n}
\newcommand{\qdist}{\textup{qdist}}
\newcommand{\DF}{\mathcal{D}(F)}
\newcommand{\infinitedimensional}{infinite-dimensional\xspace}
\newcommand{\nullspace}{nullspace\xspace}
\newcommand{\nullspaces}{\nullspace{}s\xspace}

\newenvironment{myenumerate}{%
\begin{list}{(\alph{enumcount})}
{\setcounter{enumcount}{1}\usecounter{enumcount}
\setlength{\topsep}{1mm}
\setlength{\itemsep}{0mm}
\setlength{\listparindent}{4mm}
\setlength{\parsep}{0mm}
\setlength{\labelwidth}{0mm}
\setlength{\labelsep}{3mm}
\setlength{\itemindent}{3mm}
\setlength{\leftmargin}{0mm}
}}{\end{list}}

\begin{document}

\maketitle
\newcounter{enumcount}
\renewcommand{\theenumcount}{(\alph{enumcount})}

\begin{abstract}
We consider different concepts of well-posedness and ill-posedness and their relations  for solving nonlinear and linear operator equations in Hilbert spaces.
First, the concepts of Hadamard and Nashed are recalled which are appropriate for linear operator equations. For nonlinear operator equations, stable respective unstable solvability is considered, and the properties of local well-posedness and ill-posedness are investigated.
Those two concepts consider stability in image space and solution space, respectively,
and both seem to be appropriate concepts for nonlinear operators which are not onto and/or not, locally or globally, injective.
Several example situations for nonlinear problems are considered, including the prominent autoconvolution problems and other quadratic equations in Hilbert spaces.

It turns out that for linear operator equations, well-posedness and ill-posedness are global properties valid for all possible solutions, respectively.
The special role of the nullspace is pointed out in this case.

Finally, non-injectivity also causes differences in the saturation behavior of Tikhonov and Lavrentiev regularization of linear ill-posed equations. This is examined at the end of this study.
\end{abstract}

\section{Introduction}
\label{sec:s1}
The aim of this study is to discuss different concepts of well-posedness and ill-posedness for operator equations
\begin{equation} \label{eq:nonlinopeq}
F(x)=y\,
\end{equation}
where $F: \mathcal{D}(F) \subseteq X \to Y$ is an operator mapping with domain of definition $\mathcal{D}(F)$ between (in general \infinitedimensional) Hilbert spaces $X$ and $Y$ with norms $\|\cdot\|$ and inner products $\langle \cdot,\cdot\rangle$.
For the special case that $F:=A \in \mathcal{L}(X,Y)$ is a bounded linear operator with $\mathcal{D}(F)=X$  we write the corresponding linear operator equation as
\begin{equation} \label{eq:linopeq}
Ax=y\,.
\end{equation}
In the sequel, we denote by $\xdag \in \mathcal{D}(F)$ solutions to (\ref{eq:nonlinopeq}) and (\ref{eq:linopeq}), respectively, where $y \in Y$ characterizes
the exact right-hand side. The goal for the treatment of such operator equations consists in their \emph{stable approximate solution}, where we assume that instead of $y$ only noisy data $\yd \in Y$ are available obeying the
deterministic noise model
\begin{equation} \label{eq:noise}
\|y-\yd\| \leq\delta
\end{equation}
with noise level $\delta>0$. In this context, Hadamard's classical concept plays a prominent role, which assumes that for \emph{well-posedness in the sense of Hadamard} all three conditions (cf.~\cite{Hadamard23})
\begin{itemize}
\item[(i)] For all $y \in Y$ there exists a solution $\xdag$ (existence condition).
\item[(ii)] The solutions are always uniquely determined (uniqueness condition).
\item[(iii)] The solutions depend stably on the data $y$, i.e.~small perturbations in the right-hand side lead to only small errors in the solution (stability condition).
\end{itemize}
Otherwise the corresponding operator equation is \emph{ill-posed in the sense of Hadamard}.

\bigskip

Hadamard's well-posedness concept in its entirety can only be of importance for linear equations (\ref{eq:linopeq}). In the nonlinear case, the range $F(\mathcal{D}(F))$ of $F$ will rarely coincide with $Y$ such that $(i)$
is suspicious. With respect to $(ii)$ we have to distinguish the case of injective $F$ with well-defined inverse $F^{-1}:F(\mathcal{D}(F)) \to \mathcal{D}(F)$, where $(ii)$ is fulfilled, and the case of a non-injective operator $F$, where $(ii)$ is violated. In the former case $F^{-1}(y)$ is single-valued for all $y \in F(\mathcal{D}(F))$, whereas in the latter case the symbol $F^{-1}(y)$ can be used for characterizing the set of preimages to $y$.
In the following, we also distinguish these two cases to define the stability condition $(iii)$ in a more precise manner.

For operator equations (\ref{eq:nonlinopeq}) with \emph{nonlinear}
operator $F$, the literature on
inverse problems uses concepts of well-posedness and ill-posedness
mostly in a rather rough manner,
because in contrast to the linear case (cf.~\cite{Nashed86} and
Definition~\ref{def:Nashed})
the closedness of the range $F(\mathcal{D}(F))$ of the forward
operator $F$ does not serve as an
appropriate criterion. In \cite[p.~53]{Scherzetal09},
the authors simply mention that
``it is common to consider inverse problems to be ill-posed in the
sense that the solution
(provided it exists) is unstable with respect to data perturbations.''
Now stable and unstable
behavior for nonlinear forward operators $F$ is not only a local
property and can change from
point to point, but one has also to distinguish local properties in
the image space (cf.~Definition~\ref{def:stability} below)
and in the solution space (cf.~Definition~\ref{def:posed} below). This
clear distinction is one of the aims of this
study.

Let us first consider the local stability behavior in the image space.
For forward operators $F$
injective on $\mathcal{D}(F)$, stability at some point $y \in
F(\mathcal{D}(F))$ means that the single-valued inverse operator
$F^{-1}: F(\mathcal{D}(F)) \subseteq Y \to \mathcal{D}(F) \subseteq X$
is continuous at $y \in F(\mathcal{D}(F))$,
which is characterized by the fact that for every sequence $\{y_n\}_{n=1}^\infty \subset
F(\mathcal{D}(F))$ with $\lim \limits_{n \to \infty}
\|y_n-y\|=0\,$ we have that
\begin{equation} \label{eq:injlimqdist}
  \lim \limits _{n \to \infty} \|F^{-1}(y_n)-F^{-1}(y)\| =0\,.
  \end{equation}
In case of a non-injective operator $F$, the inverse $F^{-1}$ is a
set-valued mapping and stability or instability are based on
continuity concepts
of set-valued mappings. In the literature, the Painlev\`{e}--Kuratowski set
convergence
and the Pompeiu--Hausdorff set convergence are preferred (cf., e.g.,
\cite[Section~3B]{DonRock09}
and for \infinitedimensional spaces \cite{AubFra90}). The symmetric
Hausdorff distance, however,
is problematic for modeling the approximate solution of inverse
problems if on the one hand all solutions from the
set $F^{-1}(y)$ are equally suitable and if one has on the other hand
no influence on the choice of the specific
approximate solution from the set $F^{-1}(y_n)$. Then the
non-symmetric quasi-distance ${\rm qdist}(\cdot,\cdot)$,
which was already exploited in
\cite[Section 1.6.3, p.~40]{IvaVasTan78}
and \cite{Hof86},
seems to be an appropriate stability measure
in the non-injective case. Therefore, we suggest the following definition.

\begin{definition} \label{def:stability}
We call the operator equation (\ref{eq:nonlinopeq}) \emph{stably
solvable} at the point $y \in F(\mathcal{D}(F))$ if we have for every
sequence $\{y_n\}_{n=1}^\infty \subset F(\mathcal{D}(F))$ with $\lim
\limits_{n \to \infty}
\|y_n-y\|=0\,$ that
\begin{equation} \label{eq:limqdist}
  \lim \limits _{n \to \infty} {\rm qdist}(F^{-1}(y_n),F^{-1}(y)) =0\,,
  \end{equation}
  where
\begin{equation} \label{eq:qdist}
  {\rm qdist}(U,V):= \sup \limits_{u \in U} \inf \limits _{v \in V} \|u-v\|
  \end{equation}
denotes the quasi-distance between the sets $U$ and $V$.
$ \quad \vartriangle$
\end{definition}

\begin{rem} \label{rem:rem1}
If $F$ is injective on $\mathcal{D}(F)$, then the limit condition
(\ref{eq:limqdist}) means that for every sequence
$\{y_n\}_{n=1}^\infty \subset F(\mathcal{D}(F))$ with $\lim_{n
\to \infty}
\|y_n-y\|=0\,$ we have that (\ref{eq:injlimqdist}) is satisfied.
We mention that the violation of stable solvability at $y \in
F(\mathcal{D}(F))$ in the sense of our Definition~\ref{def:stability}
is equivalent to local ill-posedness in this point
in the sense of \cite[Def.~2.2]{LuFle12}, with $ M = X$ and $S(y) =  F^{-1}(y) $ there.
\quad $ \vartriangle $
\end{rem}
\section{Stability and Nashed's ill-posedness concept for linear problems}
\label{sec:s2}

Now we turn our considerations to linear operator equations (\ref{eq:linopeq}) with a bounded linear operator $A \in \mathcal{L}(X,Y)$ mapping between the Hilbert spaces $X$ and $Y$. Hadamard's existence condition $(i)$ takes place if and only if $A$ is \emph{surjective}, i.e.~$\mathcal{R}(A)=Y$. The uniqueness condition $(ii)$ is valid iff $A$ is \emph{injective}, i.e.~$\mathcal{N}(A)=\{0\}$. In this case of a trivial nullspace $\mathcal{N}(A)$ the stability condition $(iii)$
of Hadamard holds iff the linear inverse operator $A^{-1}: \mathcal{R}(A) \to X$ is bounded. For bounded $A^{-1}$ by applying Definition~\ref{def:stability} with $F:=A$ and $\mathcal{D}(F)=X$ we have that (\ref{eq:linopeq})
is \emph{stably solvable everywhere} on $\mathcal{R}(A)$. Conversely the equation is \emph{stably solvable nowhere} if $A^{-1}$ is unbounded. For non-trivial \nullspaces $\mathcal{N}(A)\not=\{0\}$ instead of $A^{-1}$  the Moore--Penrose inverse $A^\dagger: \mathcal{R}(A)\oplus \mathcal{R}(A)^\perp \subseteq Y \to X$ has to be considered in order to make stability assertions. In this context, the conceptional suggestions of M.~Z.~Nashed in \cite{Nashed86} are helpful, and following this concept  we define well-posedness and ill-posedness of a linear operator equation (\ref{eq:linopeq}) based on the fact whether the range $\mathcal{R}(A)$ of $A$ is closed or not, which is equivalent to the question whether the Moore--Penrose inverse $A^\dagger$ is a bounded or an unbounded linear operator. Accordingly, we distinguish the ill-posedness types I and~II.

\begin{definition} \label{def:Nashed}
We call a linear operator equation (\ref{eq:linopeq}) with a bounded linear operator $A$, mapping between the Hilbert spaces $X$ and $Y$, \emph{well-posed in the sense of Nashed} if  the range $\mathcal{R}(A)$ of $A$ is a closed subset of $Y$, consequently \emph{ill-posed in the sense of Nashed} if the range is not closed, i.e.~$\mathcal{R}(A) \not= \overline{\mathcal{R}(A)}^{Y}$.
In the ill-posed case, the equation (\ref{eq:linopeq}) is called \emph{ill-posed of type I} if the range $\mathcal{R}(A)$ contains an \infinitedimensional closed subspace, and \emph{ill-posed of type~II} otherwise.
\quad $ \vartriangle $
\end{definition}
\begin{rem} \label{rem:rem2}
Ill-posedness in the sense of Nashed requires that
\begin{equation}\label{eq:infdim}
{\rm dim}\,\mathcal{R}(A)=\infty\,,
\end{equation}
and under (\ref{eq:infdim}) the equation (\ref{eq:linopeq}) ill-posed in the sense of Nashed of type~II if and only if $A$ is compact. Well-posedness, however, does not exclude the case of non-injective $A$ possessing non-trivial
\nullspaces $\mathcal{N}(A)$.
We note that an analog to Definition~\ref{def:Nashed} in Banach spaces, but only for injective $A$, has been discussed in the context of $\ell^1$-regularization in \cite{FHV15}.
\quad $ \vartriangle $
\end{rem}

\begin{proposition} \label{pro:pro1}
If the linear operator equation (\ref{eq:linopeq}) is well-posed in the sense of Nashed, then the equation is stably solvable everywhere on $\mathcal{R}(A)=\overline{\mathcal{R}(A)}^Y$, which means that
for every sequence $\{y_n\}_{n=1}^\infty \subset \mathcal{R}(A)$ with $\lim_{n \to \infty} \|y_n-y\|=0$ and $y \in \mathcal{R}(A)$ the limit condition (\ref{eq:limqdist}) with $F:=A$ and $\mathcal{D}(F):=X$ holds true. If (\ref{eq:linopeq}) is ill-posed in the sense of Nashed, the equation is stably solvable nowhere.
\end{proposition}
\begin{proof}
For arbitrary $y_n,y \in \mathcal{R}(A)$ with $\lim \limits_{n \to \infty} \|y_n-y\|=0$ and
$$F^{-1}(y)=\{x \in X:\,x=A^\dagger y+x_0,\;x_0 \in \mathcal{N}(A)\}$$ as well as $$F^{-1}(y_n)=\{x \in X:\,x=A^\dagger y_n+\tilde x_0,\;\tilde x_0 \in \mathcal{N}(A)\} \quad (n \in \mathbb{N}), $$ we can argue for every
$x_n \in F^{-1}(y_n)$ as follows: Since $A^\dagger y_n - A^\dagger y$ is orthogonal to $\mathcal{N}(A)$, the equality
\begin{equation} \label{eq:inflin}
\min \limits_{x \in F^{-1}(y)} \|x_n-x\| = \|A^\dagger y_n - A^\dagger y\|
\end{equation}
is valid. In particular, for (\ref{eq:linopeq}) well-posed in the sense of Nashed we have $\|A^\dagger\|_{\mathcal{L}(Y,X)} < \infty$ and hence
$${\rm qdist}(F^{-1}(y_n),F^{-1}(y))=\min \limits_{x \in F^{-1}(y)} \|x_n-x\| \le \|A^\dagger\|_{\mathcal{L}(Y,X)} \|y_n-y\| \to 0 \quad \mbox{as} \quad n \to \infty \,.$$
Consequently, the equation is stably solvable everywhere on $\mathcal{R}(A)$. On the other hand, for (\ref{eq:linopeq}) ill-posed in the sense of Nashed we have that $A^\dagger$ is unbounded. Hence, there exist
sequences $\{y_n\}_{n=1}^\infty$ in the range of $A$ such that $\|A^\dagger y_n - A^\dagger y\| \not\to 0$ although $\|y_n-y\| \to 0$  as $n \to \infty$. In this case,  the equation is stably solvable nowhere
on $\mathcal{R}(A)$.
\end{proof}

\section{Local well-posedness and ill-posedness}
\label{sec:s3}
\subsection{Basic notations}
In the nonlinear case, there occurs in general a locally varying behavior of solutions, moreover often an overlap of instability of solutions with respect to small data perturbations and the existence of distinguished solution branches. As a consequence, O.~Scherzer and the first author have suggested in \cite[Def.~1.1]{HofSch98} (see also
\cite[Def.~1.5]{Hofm00} or \cite[Def.~3.15]{Schusterbuch12}) a local concept in the sense of the following Definition~\ref{def:posed}. Recently, A.~Kirsch and A.~Rieder have taken advantage of this concept in \cite{KiRi14,KiRi16} for analyzing
inverse problems in seismic tomography, electrodynamics and elasticity.

\begin{definition}  \label{def:posed}
The operator equation (\ref{eq:nonlinopeq}) is called \emph{locally well-posed} at the solution $\xdag \in \mathcal{D}(F)$ if there is a closed ball $\mathcal{B}_r(\xdag)$ with radius $r>0$ and center $\xdag$ such that for
every sequence $\{x_n\}_{n=1}^\infty \subset \mathcal{B}_r(\xdag)\cap \mathcal{D}(F)$ \linebreak
the convergence of images $\lim_{n \to \infty} \|F(x_n)-F(\xdag)\|=0$ implies the convergence of the preimages
$\lim_{n \to \infty} \|x_n-\xdag\|=0$. Otherwise equation (\ref{eq:nonlinopeq}) is called \emph{locally ill-posed} at $\xdag$.
\quad $ \vartriangle $
\end{definition}
\begin{rem} \label{rem:rem3}
Local well-posedness at $\xdag$ in the sense of Definition~\ref{def:posed} requires \emph{local injectivity}, which means that $\xdag$ is the only
solution in $\mathcal{B}_r(\xdag)\cap \mathcal{D}(F)$ and hence $\xdag$ is an isolated solution of the operator equation. This often provokes criticism of Definition~\ref{def:posed}, but the underlying idea of this definition is that a
really existing physical quantity $\xdag$ is the unique solution of (\ref{eq:nonlinopeq}) in the ball $\mathcal{B}_r(\xdag)\cap \mathcal{D}(F)$  and can be recovered exactly when
the measurement process may be taken arbitrarily precise, i.e.~when $\delta \to 0$ can be implemented. The idea does not exclude the case that further branches of solutions to (\ref{eq:nonlinopeq}) exist
in $\mathcal{D}(F)$ outside of the ball. For another instability concept mixing solvability and local ill-posedness we refer to \cite[p.~1]{LorWor13}.
\quad $ \vartriangle $\end{rem}

Directly from the Definitions~\ref{def:stability}  and \ref{def:posed} we have the following implication. Notice that the converse implication in Proposition \ref{pro:pro0} does not hold, in general.

\begin{proposition} \label{pro:pro0}
Let the operator $F$ be locally injective at $ \xdag \in \DF $. If the operator equation (\ref{eq:nonlinopeq}) is stably solvable at $y=F(\xdag) $, then this equation is locally well-posed at $\xdag$.
\end{proposition}
%

The nonlinear operator equation (\ref{eq:nonlinopeq}) can be stably solvable at some points $y=F(\xdag),\;\xdag \in \mathcal{D}(F)$, and not stably solvable at other points of the range.
Similarly, (\ref{eq:nonlinopeq}) can be locally well-posed at some points $\xdag \in \mathcal{D}(F)$ and locally ill-posed at other points of the domain of definition.
The following three examples illustrate such behavior.

\subsection{Three examples}
\begin{example}[one-dimensional example]
Let $ X =Y:= \mathbb R$ and consider the nonlinear mapping $F: \mathbb R \to \mathbb R$ defined as
$$F(x):=\frac{x^2}{1+x^4}\,. $$
Then the corresponding equation  (\ref{eq:nonlinopeq}) is locally well-posed for all $\xdag \in \mathbb R$. However, the equation is evidently not stably solvable at $y=F(0)=0$,
because we have, for $y_n>0$ tending to zero as $n \to \infty$,  $\lim_{n \to \infty}{\rm qdist}(F^{-1}(y_n),F^{-1}(0))=+\infty$. On the other hand, the equation is stably solvable for all other range points
$y>0$.
\quad $ \vartriangle $\end{example}
In \cite{FlemmingHabil17} and earlier in \cite{Flemming14}, J.~Flemming has comprehensively studied \emph{quadratic} operators $F: X \to Y$ between separable Hilbert spaces $X$ and $Y$, where the nonlinear operator $F(x):=B(x,x) \in Y$ defined for all $x \in X$ is based on a bounded bilinear mapping $B: X \times X \to Y$ with $\|B(u,v)\| \le c \,\|u\|\,\|v\|$ for all $u,v \in X$ and some constant $c>0$.
The class of nonlinear equations (\ref{eq:nonlinopeq}) with quadratic operators $F$ and $\mathcal{D}(F)=X$ is rather close to linear equations (\ref{eq:linopeq}).
In the sequel we examine
the properties local well-posedness and stable solvability by means of
two examples with quadratic examples.

We first present an example of a quadratic operator $ F $ such that the corresponding operator equation (\ref{eq:nonlinopeq}) is locally well-posed at some points, and local ill-posedness may occur at other points of the domain of definition.
%
\begin{example}[selfintegration-weighted identity operator]
\label{ex:quadratic_example_1}
Let $ X =Y:= L^2_{\mathbb R}(0,1)$ (Hilbert space of real-valued square integrable functions over the unit interval $(0,1)$), and let
the quadratic operator $ F : L^2_{\mathbb R}(0,1) \to L^2_{\mathbb R}(0,1) $ be given by
$$
[F(x)](s) := \mypsi(x)\, x(s),\;\;s \in (0,1),\;\; \ \textup{ for } \ x \in X, \ \textup{where }\;\; \mypsi(x) := \int_0^1 x(t) dt.
$$
This operator is injective on $ X\backslash N $, but at each point of $ N $ it
is not locally injective, where $ N = F^{-1}(0) $. In addition, the corresponding operator equation (\ref{eq:nonlinopeq}) is locally ill-posed everywhere on $ N $, locally well-posed everywhere on $ X\backslash N $, and stably solvable everywhere on $ F(X) $.
Details are given in Proposition \ref{prop:quadratic_example_1}
and its proof, which are presented in the appendix.
\quad $ \vartriangle $
\end{example}
There also exist quadratic operators $ F $ such that corresponding operator equation (\ref{eq:nonlinopeq}) is stably solvable at some points of the range, and
at other points it is not stably solvable as the following example shows.
\begin{example}[quadratic combination of two linear operators with closed range respective non-closed range]
\label{ex:quadratic_example_2}
Let $ X $ be an infinite-dimensional, separable real Hilbert space with orthonormal basis $ \{\uk\}_{k=1}^\infty $. Consider the damped shift operator
$ S: X \to X $ and the partial isometry operator $ T: X \to X $ given by
\begin{align*}
Sx := \skp{x}{\uk[1]} \uk[1] + \sum_{k=3}^\infty \sigma_k \skp{x}{\uk} \uk[k+1],
\qquad Tx := \sum_{k=2}^\infty \skp{x}{\uk} \uk, \qquad x \in X,
\end{align*}
where $ 0 \neq \sigma_k \in \mathbb{R} $ satisfies $ \sigma_{k} \to 0 $ as $ k \to \infty $.
Thus $ S $ and $ T $ are bounded linear operators, with $ S $ being compact.
%
Consider the following quadratic operator $ F: X \to X $:
\begin{align*}
F(x) := B(x,x), \quad \textup{where }
B(x, y) := \skp{x}{\uk[1]} Sy +  \skp{x}{\uk[2]} Ty, \quad x,y \in X.
\end{align*}
%
For the operator equation (\ref{eq:nonlinopeq}) with $X=Y$
corresponding to this operator $ F $, we have stable solvability at some points in the range of $ F $, as well as unstable solvability at other points in the range.
Details can be found in Proposition \ref{prop:quadratic_example_2} and its proof in the appendix.
\quad $ \vartriangle $
\end{example}
%
%
\subsection{Autoconvolution problems}
\label{sec:s5}
%
We will mention by means of two example situations some corresponding details concerning \emph{autoconvolution} operators $F(x):=x \star x$ which form an interesting subclass of quadratic operators, for which the
associated equations  (\ref{eq:nonlinopeq}) possess numerous applications in natural sciences and engineering.

\begin{example} \label{ex:complauto}
{\rm Let $X:=L^2_{\mathbb C}(0,1)$ and $Y:=L^2_{\mathbb C}(0,2)$
(Hilbert spaces of complex-valued square integrable functions over the intervals $(0,1)$ and $(0,2)$, respectively). Then the autoconvolution operator $F$ with $\mathcal{D}(F)=X$ generated from the bilinear mapping $B(u,v)=u \star v$ attains the form
\begin{equation} \label{eq:complauto}
[F(x)](s):=\begin{cases} \quad \int \limits _0^s \;\; x(s-t)\,x(t)\,dt,\qquad 0 \le s \le 1,\\ \quad \int \limits _{s-1}^1 \; x(s-t)\,x(t)\,dt,\qquad 1 < s \le 2.  \end{cases}
\end{equation}
As a consequence of Titchmarsh's convolution theorem , the corresponding operator equation (\ref{eq:nonlinopeq}) possesses for arbitrary $y =F(\xdag), \;\xdag \in X,$ the solution set $F^{-1}(y)=\{\xdag,-\xdag\}$. Moreover, this equation is locally ill-posed everywhere as Example~3.1 in \cite{GHBKS14} indicates. Evidently, the equation is also stably solvable nowhere if $\mathcal{D}(F)=X$, but stable solvability can be achieved by an appropriate restriction of the domain, for example to
$\mathcal{D}(F):=\{x \in H_\mathbb{C}^\alpha(0,1):\,\|x\|_{H_\mathbb{C}^\alpha}\le c\}$ with some $\alpha > 0$ and $0<c<\infty$ (cf.~\cite[Theorem~4.2]{GHBKS14} and \cite[Proposition~2.3]{FleiHof96}).\quad $ \vartriangle $}\end{example}
\begin{example} \label{ex:realauto}
{\rm Now we turn to  the space situation $X=Y:=L^2_{\mathbb R}(0,1)$.
Here, the autoconvolution operator $F$ with
$$\mathcal{D}(F):=\{x \in X:\,x(t) \ge 0\;\; \mbox{a.e.~on}\;\;t \in (0,1)\}$$ can be written as
\begin{equation} \label{eq:realauto}
[F(x)](s):= \int \limits _0^s \;\; x(s-t)\,x(t)\,dt,\qquad 0 \le s \le 1.
\end{equation}
In this case, for $y \in F(\mathcal{D}(F))$ the solution $\xdag$ is uniquely determined if and only if there is no $\varepsilon>0$ such that $y(s)=0,\;0 \le s \le \varepsilon$ (cf.~\cite[Theorem~1]{GorHof94}), but for all $\xdag \in \mathcal{D}(F)$ the operator equation is locally ill-posed at $\xdag$ (cf.~\cite[Example~2.1]{BuerHof15}).
\quad $ \vartriangle $}\end{example}

\subsection{The linear case}
\textcolor{black}{For linear operator equations (\ref{eq:linopeq}), local well-posedness at some points and local ill-posedness at other points may not occur
as the following proposition shows.
}

\begin{proposition} \label{pro:pro2}
The linear operator equation (\ref{eq:linopeq}) is locally well-posed (cf.~Definition~\ref{def:posed} with $F:=A$ and $\mathcal{D}(F):=X$) everywhere on $X$ if the equation is on the one hand well-posed in the sense of Nashed (i.e.,~$\mathcal{R}(A)=\overline{\mathcal{R}(A)}^Y$)\, and if on the other hand the \nullspace of $A$ is trivial (i.e.,~$\mathcal{N}(A)=\{0\})$.
Otherwise, (\ref{eq:linopeq}) is locally ill-posed everywhere on $X$.
 \end{proposition}
\begin{proof}
Directly from the Definition~\ref{def:posed} we find that (\ref{eq:linopeq}) is locally ill-posed everywhere if $\mathcal{N}(A)\not=\{0\}$. In the case $\mathcal{N}(A)=\{0\}$ and for linear $A$, local well-posedness at $\xdag$ is valid if and only if the implication
$$\|A\,\Delta_n\| \to 0 \quad \Longrightarrow \quad \|\Delta_n\| \to 0 \quad \mbox{as} \quad n \to \infty$$
holds whenever $\|\Delta_n\| < r$, where we have in mind $\Delta_n:=x_n-\xdag$. This implication, however, is valid if and only if the inverse operator $A^{-1}: \mathcal{R}(A) \to X$ is bounded, which just characterizes the situation of a closed range $\mathcal{R}(A)=\overline{\mathcal{R}(A)}^Y$.
\end{proof}

Consequently, for linear operator equations well-posedness and ill-posedness, respectively, are \emph{global properties} valid for all possible solutions $\xdag \in X$.
However, by comparing Propositions~\ref{pro:pro1} and
\ref{pro:pro2} it becomes clear that the concepts of stable solvability (cf.~Definition~\ref{def:stability}) and of local well-posedness (cf.~Definition~\ref{def:posed}) differ with respect to the \nullspace property of $A$. In the following section, we will motivate that the consideration
of both concepts is justified with respect to the error analysis of different \emph{regularization methods} (cf., e.g.,~\cite{EHN96,IvaVasTan78,Scherzetal09,TikArs77}).

\section{Saturation of linear Tikhonov versus Lavrentiev regularization}
\label{sec:s4}

The most prominent method for the stable approximate solution of the linear operator equation (\ref{eq:linopeq}) is the classical \emph{Tikhonov regularization}, where for some \emph{regularization parameter} $\alpha>0$
Tikhonov-regularized solutions $\xad$ have the form
\begin{equation} \label{eq:Tik}
\xad=(A^*A+\alpha I)^{-1}A^*y^\delta\,.
\end{equation}
For $\delta \to 0$, convergence of $\xad$ to the uniquely determined \emph{minimum-norm solutions} $\xdag=A^\dagger y$ can be ensured if the parameter choice $\alpha=\alpha(\delta,y^\delta)$ is appropriate. One tool for the error analysis
ist the \emph{maximal best possible error}
\begin{equation} \label{eq:maxbest}
P^\delta(\xdag):=  \sup \limits_{y^\delta \in Y:\,\|A \xdag-y^\delta\| \le \delta} \,   \inf \limits_{\alpha>0} \|\xad-\xdag\|\,.
\end{equation}
Well-known convergence rates results yield for the classical Tikhonov regularization
$$ P^\delta(\xdag)=  \mathcal{O}(\delta^{2/3}) \quad \mbox{as} \quad \delta \to 0$$
under the \emph{source condition}
$$\xdag=A^*A v, \quad v \in X\,,$$
but also the following \emph{saturation result} is well known.

\begin{proposition} \label{pro:pro3}
If the linear operator equation (\ref{eq:linopeq}) is ill-posed in the sense of Nashed (i.e.,~$\mathcal{R}(A)\not=\overline{\mathcal{R}(A)}^Y$), then we have the implication
$$P^\delta(\xdag)=  o(\delta^{2/3}) \quad \mbox{as} \quad \delta \to 0 \quad \Longrightarrow \quad \xdag=0\,. $$
\end{proposition}
If (\ref{eq:linopeq}) is ill-posed in the sense of Nashed of type~II ($A$ compact), then this result was proven by Groetsch \cite[Theorem~3.2.4]{Groe84}, whereas Neubauer's result in \cite[Theorem~3.1]{Neubauer97} also includes the case of ill-posedness of type~I.

\begin{rem} \label{rem:rem4}
If (\ref{eq:linopeq}) is well-posed in the sense of Nashed (i.e.,~$\mathcal{R}(A)=\overline{\mathcal{R}(A)}^Y$ and hence $\|A^\dagger\|_{\mathcal L(Y,X)}<\infty$), we have
$$P^\delta(\xdag)=  \mathcal{O}(\delta) \quad \mbox{as} \quad \delta \to 0$$
for all minimum-norm solutions $\xdag=A^\dagger y$, regardless of whether or not the \nullspace $\mathcal{N}(A)$ is trivial. Namely, we have then for all $y^\delta$ obeying (\ref{eq:noise})
$$\inf \limits_{\alpha>0} \|\xad-\xdag\| \le \lim \limits_{\alpha \to 0} \|\xad-\xdag\|=\|A^\dagger y^\delta-A^\dagger y\| \le \|A^\dagger\|_{\mathcal L(Y,X)}\,\delta.
\quad \vartriangle
$$
\end{rem}

\medskip

If $X=Y$ and the linear operator $A \in \mathcal{L}(X,X)$ in the operator equation (\ref{eq:linopeq}) is \emph{accretive}, which means that $ \re \langle Ax,x \rangle \ge 0$ for all $x \in X$, then instead of the Tikhonov regularization the simpler \emph{Lavrentiev regularization} with regularized solutions
\begin{equation} \label{eq:Lav}
\xad=(A+\alpha I)^{-1}y^\delta
\end{equation}
can be used (cf.,~e.g.,~\cite{BoHo16,Plato95,Tautenhahn02}). Well-known convergence rates results yield here
$$ P^\delta(\xdag)=  \mathcal{O}(\delta^{1/2}) \quad \mbox{as} \quad \delta \to 0$$
under the \emph{source condition}
$$\xdag=A w, \quad w \in X\,.$$
The following proposition is a reformulation of a saturation result published recently by the second author, see \cite[Theorem~5.1]{Plato17}.
\begin{proposition} \label{pro:pro4}
If the linear operator $A \in \mathcal{L}(X,X)$ in equation (\ref{eq:linopeq}) is accretive (i.e.,~$\re \langle Ax,x\rangle  \ge 0\;\forall x \in X$), but not surjective (i.e.,~$\mathcal{R}(A) \not=X$), then
we have the implication
$$P^\delta(\xdag)=  o(\delta^{1/2}) \quad \mbox{as} \quad \delta \to 0 \quad \Longrightarrow \quad \xdag=0\,. $$
\end{proposition}
Notice that symmetry of the operator $ A $ is not required in the preceding proposition, thus no spectral calculus of $ A $ is available in this setting, in general. As a consequence, the proof of Proposition \ref{pro:pro4} requires techniques which differ substantially from those used in the proof of Proposition \ref{pro:pro3}.

\begin{rem} \label{rem:rem5}
The conditions for the saturation result of Proposition~\ref{pro:pro4} concerning the Lavrentiev regularization coincide with the requirement that equation (\ref{eq:linopeq}) is locally ill-posed everywhere on $X$ (cf.~Proposition~\ref{pro:pro2}).
This is a consequence of the identity $$\overline{\mathcal{R}(A)}^X \oplus\, \mathcal{N}(A) =X\,,$$ which is valid for accretive $A$. Thus, saturation for Lavrentiev regularization deviates from that of Tikhonov regularization, where the saturation result of Proposition~\ref{pro:pro3} requires no \nullspace conditions. This deviation is based on the fact that the adjoint operator $A^*$ occurs in formula (\ref{eq:Tik}) but not in (\ref{eq:Lav}) and that
we have for general $A \in \mathcal{L}(X,Y)$ the identity $$\overline{\mathcal{R}(A)}^Y \oplus\, \mathcal{N}(A^*) =Y\,.
\quad \vartriangle
$$
\end{rem}

\section{Appendix: details on two quadratic examples}
In this appendix we give some details related with Examples \ref{ex:quadratic_example_1}
and \ref{ex:quadratic_example_2} presented in Section \ref{sec:s3}.
\begin{proposition}
\label{prop:quadratic_example_1}
The quadratic operator $ F $ from Example \ref{ex:quadratic_example_1}
is locally injective at each point of $ X\backslash N $,
where $ N = F^{-1}(0) $.
The corresponding operator equation (\ref{eq:nonlinopeq}) is
locally ill-posed everywhere on $ N $, locally well-posed everywhere on $ X\backslash N $, and stably solvable everywhere on $ F(X) $.
%
\end{proposition}
\begin{proof}
We first state some elementary properties of the operator $ F $:
\begin{itemize}
\item
 The range $ F(X) $ of $ F $ is given by
$ F(X) = \{y \in X:\, y \equiv 0 \,\textup{ or } \,\mypsi(y) > 0 \} $.
\item
We have $ F^{-1}(0) = N := \mathcal{N}(\mypsi) $ (the \nullspace of $\mypsi $).

\item
 For each $ y \in X $ with $ \mypsi(y) > 0 $ we have
$ F^{-1}(y) = \big\{ x^\pm := \pm \frac{y}{\sqrt{\mypsi(y)}} \big\} $.
\end{itemize}
Here, $ \varphi $ means integration, cf. Example \ref{ex:quadratic_example_1}.
\begin{myenumerate}
\item
\label{enum:quadratic_example_1a}
Local injectivity of the operator equation $ F(x) = y $ everywhere on
$ X\backslash N $ immediately follows from the given representation of $ F^{-1}(y) $.

\item
\label{enum:quadratic_example_1b}
We next show that the operator equation (\ref{eq:nonlinopeq}) with such $F$ is stably solvable everywhere on $ F(X) $. For this purpose let $ y, \, y_n \in F(X) $
with $ \lim_{n\to \infty} \Vert y_n - y \Vert = 0 $. If $ y \neq 0 $, then
$ \lim_{n\to \infty} \mypsi(y_n) = \mypsi(y) > 0 $,
and thus $ \mypsi(y_n) > 0 $ for $ n $ large enough. For
$ x_n^\pm := \pm \mypsi(y_n)^{-1/2} y_n$ with sufficiently large $n$, we then have
\begin{align*}
\textup{qdist}(F^{-1}(y_n),F^{-1}(y))
\le \max\{ \Vert x_n^+ - x^+ \Vert, \Vert x_n^- - x^- \Vert \} \to 0
\;\; \textup{ as }\;\; n \to \infty.
\end{align*}
We next consider the case $ y \equiv 0 $. If $ y_n = 0 $, then we have
$ \textup{qdist}(F^{-1}(y_n),F^{-1}(y)) = 0 $. Without loss of generality we thus may assume
$ \mypsi(y_n) > 0 $ for each $ n $.
For $ z_n^\pm :=
x_n^\pm - \mypsi(x_n^\pm) 1 $, with $ x_n^\pm $ as above, we have $ \mypsi(z_n^\pm) = 0 $ and thus
$ z_n^\pm \in F^{-1}(0) = N $, and
\begin{align*}
\textup{qdist}(F^{-1}(y_n),F^{-1}(0))
\le \max\{ \Vert x_n^+ - z_n^+ \Vert, \Vert x_n^- - z_n^- \Vert \}
= \sqrt{\mypsi(y_n)}
 \to 0
\ \textup{ as } n \to \infty.
\end{align*}
Thus the corresponding operator equation (\ref{eq:nonlinopeq}) is indeed stably solvable everywhere on $F(X)$.

\item
\label{enum:quadratic_example_1c}
Since $ F^{-1}(0) = N $ is a nontrivial linear subspace of $ X $, it is immediate that
at each point of $ N $, the operator $ F $ fails to be locally injective, and thus the corresponding operator equation (\ref{eq:nonlinopeq}) is
locally ill-posed everywhere on $ N $.

\item
\label{enum:quadratic_example_1d}
It now follows easily from items \ref{enum:quadratic_example_1a},
\ref{enum:quadratic_example_1b} and Proposition \ref{pro:pro0}
that the equation $ F(x) = y $ is locally well-posed on $ X\backslash N $.
%
%
%
\end{myenumerate}
\end{proof}
\begin{proposition}
\label{prop:quadratic_example_2}
The quadratic operator $ F $ from Example \ref{ex:quadratic_example_2}
is locally injective at $ x = \uk[1] $
and $ x = \uk[2] $, respectively. The corresponding operator equation (\ref{eq:nonlinopeq}) is locally ill-posed at $ x = \uk[1] $, and not stably solvable at $ y = F(\uk[1]) = \uk[1] $.
It is locally well-posed at $ x = \uk[2] $, and stably solvable at $ y = F(\uk[2]) = \uk[2] $.
\end{proposition}
\begin{proof}
In what follows we make use of the notations introduced in Example \ref{ex:quadratic_example_2}.
\begin{myenumerate}
\item
\label{it:quadratic_example_2-01}
Let
\begin{align*}
M = \textup{span}\{\uk[1], \uk[2]\},
\qquad
N = \overline{\textup{span}\{\uk \mid k = 3,4, \ldots\}}^X,
\end{align*}
which means that $ M^\perp = N $. First we state some elementary mapping properties
of the considered bilinear mapping and the related quadratic operator:
\begin{align}
& B(\uk[1], \cdot) = S, \quad F(\uk[1]) = \uk[1],
\qquad B(\uk[2], \cdot) = T, \quad F(\uk[2]) = \uk[2],
\nonumber
\\
& B(w, \cdot) = 0, \quad F(w) = 0, \quad w \in N.
\label{eq:quadratic_example_2-01}
\end{align}
In addition we have
$ B(\uk[1], \uk[2]) = B(\uk[2], \uk[1]) = 0, $  and thus
\begin{align}
F(a \uk[1] + b \uk[2]) = a^2 \uk[1] + b^2 \uk[2], \quad a, b \in \mathbb{R}.
\label{eq:quadratic_example_2_M}
\end{align}
This in particular means that $ M $ is an invariant subspace with respect to $ F $, i.e., $ F(M) \subset M $.
In addition we have
\begin{align}
F(v + w) = F(v) +B(v,w), \quad v \in M, \, w \in N,
\label{eq:quadratic_example_2-02}
\end{align}
which easily follows from
\refeq{quadratic_example_2-01}.
Notice that the right-hand side of \refeq{quadratic_example_2-02} provides an orthogonal decomposition of $ F(v + w) $, with $ F(v) \in M,\, B(v,w) \in N $.

\item
\label{it:fredholm}
As a preparation, we now consider for fixed
\begin{align*}
v = a \uk[1] + b \uk[2],  \quad a, b \in \mathbb{R}, \, \vert a \vert + \vert b \vert \neq 0,
\end{align*}
the restricted linear operator
\begin{align}
& B(v,\cdot)_{\vert N} = a S_{\vert N} + b T_{\vert N}
= a S_N + b I_N: N \to N,
\label{eq:fredholm}
\end{align}
where $ S_N := S_{\vert N}, \, I_N := I_{\vert N} $.

For each $ b \neq 0 $, the linear operator
in \refeq{fredholm}
%
is an isomorphism. This follows from standard spectral theory for compact operators.
Notice that $ S_N : N \to N $ is a compact operator which has no eigenvalues.
%
For future reference we note that the identity $ \mathcal{N}(a S_N + b I_N)~=~\{0\} $ is also valid for $ a \neq 0, b = 0 $.

\item
\label{it:quadratic_example_2-injective}
For each $ 0 \neq y \in F(X) $ there exist uniquely determined two real numbers $ a, b \ge 0,\, a + b > 0 $
and elements $ \wk \in N $ for $ k =  1,\ldots, 4 $
such that
\begin{align}
F^{-1}(y) = \{ \vk + \wk ,\;k = 1,\ldots, 4 \},
\label{eq:quadratic_example_2-injective}
\end{align}
where $ \vk[1/2]= a \uk[1] \pm b \uk[2] \in M $ and $ \vk[3/4] = -a \uk[1] \pm b \uk[2] \in M $.
This follows easily from the representations
\refeq{quadratic_example_2_M} and \refeq{quadratic_example_2-02}
and the injectivity properties given in item \ref{it:fredholm}.
The elements $ \wk[1],\ldots, \wk[4] $ are given by the
affine equations
$ F(\vk) + B(\vk,\wk) = y $
for $ k = 1,\ldots,4 $.

Notice that the set \refeq{quadratic_example_2-injective} contains only two elements if
either $ a = 0 $ or $ b = 0 $ holds. Notice further
that \refeq{quadratic_example_2-injective} implies local injectivity at each element of
$ F^{-1}(y) $, in particular at the points $ x = \uk[1] \in F^{-1}(\uk[1]) $
and $ x = \uk[2] \in F^{-1}(\uk[2]) $ which in fact is the first statement of the proposition.

\item
\label{it:quadratic_example_2-illp}
We next show that the equation $ F(x) = y $ is locally ill-posed at $ x = \uk[1] $.
For this purpose consider, for any $ r > 0 $ fixed, the sequence
$ x_n = x + r \uk[n] \in \mathcal{B}_r(x),\, n = 1,2,\ldots\ $. We
obviously have $ \Vert x_n - x \Vert = r $ for each $ n $, and
\begin{align*}
F(x_n) &= F(x) + B (x, r \uk[n]) = x + r \sigma_n \uk[n+1] \to x = F(x)
\quad (n \to \infty).
\end{align*}
Thus the equation $ F(x) = y $ is indeed locally ill-posed at $ x = \uk[1] $.

\item
From item \ref{it:quadratic_example_2-illp} and Proposition \ref{pro:pro0}
it now follows directly that the operator equation $ F(x) = y $ fails to be stably solvable at $ y = \uk[1] = F(\pm\uk[1]) $.

%
%
%
%
%
%

\item
\label{it:quadratic_example_2-stable}
We next show that equation \refeq{nonlinopeq}
is stably solvable at $ y = \uk[2] $.
For this purpose consider
any sequence $ y_n \in F(X), \, n = 1,2,\ldots, $ with
$ y_n \to y = \uk[2] $  as $ n \to \infty $, and w.l.o.g we may assume that
$ y_n \neq 0 $ for each $ n $.
According to
items \ref{it:quadratic_example_2-01} -- \ref{it:quadratic_example_2-injective},
there exist representations
\begin{align}
y_n  & = v_n + w_n, \ \textup{where  } v_n \in M, \, w_n \in N,
 \nonumber \\
& \quad v_n = F(\vnk), \quad w_n = B(\vnk,\wnk),
\label{eq:quadratic_example_2-stable-1}
\\
& \quad
\vnk[1/2] = a_n \uk[1] \pm b_n \uk[2] \in M,
\quad \vnk[3/4] = -a_n \uk[1] \pm b_n \uk[2] \in M,
\quad
\wnk[1], \ldots, \wnk[4] \in N,
\nonumber
\end{align}
%
where $ a_n, b_n \ge 0 $ are real numbers satisfying $ a_n + b_n > 0 $.
Notice that in \refeq{quadratic_example_2-stable-1}, no other preimages with respect to $ F $ exist, i.e., $ F^{-1}(v_n) = \{ \vnk[1],\ldots,\vnk[4]\} $ holds,
and the points $ \wnk[1], \ldots, \wnk[4] \in N $
are uniquely determined, respectively.
We thus have
%
\begin{align*}
F^{-1}(y_n) = \{ \vnk + \wnk,\; k = 1,\ldots, 4 \}.
\end{align*}
By assumption we have
\begin{align*}
y_n = a_n^2 \uk[1] + b_n^2 \uk[2] + w_n \to
y = \uk[2]  \textup{ as }  n \to \infty,
\end{align*}
and thus
%
\begin{align}
a_n \to 0, \quad b_n \to 1, \quad
w_n \to 0
\quad \textup{ as }  n \to \infty.
\label{eq:quadratic_example_2-stable-4}
\end{align}
From the stability considerations in item \ref{it:fredholm}
it follows that
\begin{align}
\Vert \pm a_n S w + b_n w \Vert \ge c \Vert w \Vert, \quad n \ge n_0,
\ w \in N,
\label{eq:quadratic_example_2-stable-5}
 \end{align}
for some integer $ n_0 $ large enough, and with some constant $ c > 0 $ which may be chosen independently of $ n $.
The second identity in \refeq{quadratic_example_2-stable-1},
the convergence statement \refeq{quadratic_example_2-stable-4}, and
the inverse stability property \refeq{quadratic_example_2-stable-5}
then imply
$ \wnk \to 0 $ as $ n \to \infty $ for each $ 1 \le k \le 4 $.

With the notation $ x^{(k)} = (-1)^{k-1} \uk[2] $ for $ k = 1, \ldots, 4 $
we thus finally obtain
\begin{align*}
\qdist(F^{-1}(y_n), F^{-1}(y)) &\le
\max_{k=1,\ldots, 4} \Vert v^{(k)}_n + w^{(k)}_n - x^{(k)} \Vert
\\
& \le a_n + \vert 1 - b_n \vert + \max_{k=1,\ldots, 4} \Vert w^{(k)}_n \Vert
\to 0 \quad (n \to \infty).
\end{align*}
This completes the proof of stable solvability at $ y = \uk[2] $.

\item
\label{it:quadratic_example_2-wellp}
Local well-posedness of the equation $ F(x) = y $
at $ x = \uk[2] $ now follows from
part \ref{it:quadratic_example_2-stable} of this proof and Proposition
\ref{pro:pro0}.
This completes the proof of the proposition.
\end{myenumerate}
\end{proof}

\section*{Acknowledgments}
We wish to thank R.~I.~Bo\c{t} (University of Vienna) and J.~Flemming
(TU Chemnitz) for valuable hints and discussions
in the context of ill-posedness concepts.
The first author gratefully acknowledges support by the German Research Foundation (DFG) under grant HO~1454/10-1.

\bibliographystyle{plain}
\bibliography{HP}
\end{document}